\documentclass[12pt,leqno]{amsart}
\usepackage{amsmath,amssymb,amsfonts}
\usepackage[all]{xy}

\usepackage[T5,T1]{fontenc}
\usepackage{mathpazo}
\usepackage{hyperref}
\usepackage{a4wide}

\theoremstyle{plain}
\newtheorem{thm}{Theorem}[section]
\newtheorem{lem}[thm]{Lemma}
\newtheorem{cor}[thm]{Corollary}
\newtheorem{prop}[thm]{Proposition}

\newtheorem{conj}[thm]{Conjecture}

\theoremstyle{definition}
\newtheorem{defn}[thm]{Definition}

\newtheorem{rmk}[thm]{Remark}

\newcommand{\im}{{\rm im}}

\newcommand{\sC}{{\mathcal C}}

\newcommand{\sE}{{\mathcal E}}

\newcommand{\sM}{{\mathcal M}}
\newcommand{\sN}{{\mathcal N}}

\newcommand{\sR}{{\mathcal R}}

\newcommand{\F}{{\mathbb F}}

\renewcommand{\P}{{\mathbb P}}
\newcommand{\Q}{{\mathbb Q}}

\newcommand{\U}{{\mathbb U}}

\newcommand{\Z}{{\mathbb Z}}

\newcommand{\res}{{\text{\sf res}\hspace{.1ex} }}

\newcommand*{\QEDB}{\hfill\ensuremath{\square}}
\def\NDT{{\fontencoding{T5}\selectfont Nguy\~ \ecircumflex n Duy T\^an}}
 
\begin{document}
\title{Triple Massey products vanish over all fields}
\begin{abstract} We show that the absolute Galois group of any field has the vanishing triple Massey product property. Several corollaries for the structure of maximal pro-$p$-quotient  of absolute Galois groups are deduced. Furthermore, the vanishing of some higher Massey products is proved.
\end{abstract}
 \author{ J\'an Min\'a\v{c} and \NDT}
\address{Department of Mathematics, Western University, London, Ontario, Canada N6A 5B7}
\email{minac@uwo.ca}
 \address{Department of Mathematics, Western University, London, Ontario, Canada N6A 5B7 and Institute of Mathematics, Vietnam Academy of Science and Technology, 18 Hoang Quoc Viet, 10307, Hanoi - Vietnam } 
\email{duytan@math.ac.vn}
\thanks{JM is partially supported  by the Natural Sciences and Engineering Research Council of Canada (NSERC) grant R0370A01. NDT is partially supported  by the National Foundation for Science and Technology Development (NAFOSTED) grant 101.04-2014.34}

\maketitle
\section{Introduction}
Let $F$ be a field and $F_s$ a separable closure of $F$. The Galois group $G_F:= {\rm Gal}(F_s/F)$ is called the absolute Galois group of $F$. Every such Galois group is a profinite group. One may ask what special properties absolute Galois groups have among all profinite groups. This is a difficult problem, and at this moment only few properties have been found. However those discovered properties are of great interest and considerable depth. 

In the classical papers \cite{AS1,AS2} published in 1927, E. Artin and O. Schreier developed a theory of real fields, and they showed in particular that the only non-trivial finite subgroups of absolute Galois groups  are groups of order 2. More recently, in some remarkable work M. Rost and V. Voevodsky proved the Bloch-Kato conjecture, thereby establishing a very special property of Galois cohomology of absolute Galois groups. Relatively recently, two new conjectures, the Vanishing $n$-Massey Conjecture and the Kernel $n$-Unipotent Conjecture were proposed ( see \cite{MT1} and \cite{MT2}). These conjectures are based on a number of previous considerations. One motivation is coming from topological considerations. (See \cite{DGMS} and \cite{HW}.) Another motivation is a program to describe various $n$-central series of absolute Galois groups as kernels of simple Galois representations. (See \cite{Ef, EM1,EM2, MSp,Vi}.)

In this paper we shall consider only the special case of the Vanishing $n$-Massey Conjecture when $n=3$ with the exception of Section ~\ref{sec: higher Massey} when we consider $n>3$ as well.
In the papers \cite{MT1, MT2}, the Vanishing $n$-Massey Conjecture was formulated only in the case when the base field $F$ contains a primitive $p$-th root of unity. In this paper we consider a stronger version of this conjecture when there is no condition on the field $F$. 
For the sake of simplicity we shall recall below the definition of $n$-Massey products only in the case $n=3$, and we refer the reader to \cite[Sections 2 and 3]{MT1} for the more general case. See \cite[Definition 3.3]{MT1} for a definition of the vanishing $n$-fold Massey product property. Also see Section~\ref{sec: higher Massey} for reviews of some basic definitions and facts related to $n$-fold Massey products.
\begin{conj}[Vanishing $n$-Massey Conjecture]
\label{conj:vanishing n-Massey}
 Let $p$ be a prime number and $n\geq 3$ an integer. Let $F$ be a field. Then the absolute Galois group $G_F$ of $F$ has the vanishing $n$-fold Massey product property with respect to $\F_p$.
\end{conj}
In this paper we show that the conjecture is true when $n=3$. We will use the word ``triple'' instead of ``3-fold''.
\begin{thm}[= Theorem~\ref{thm:main vanishing}]
 Let $F$ be any field and $p$ a prime number. Then the absolute Galois group $G_F$ of $F$ has the vanishing triple Massey product property with respect to $\F_p$
\end{thm}

Below, we briefly recall the definition of triple Massey products and the corresponding vanishing property. 
  Let $G$ be a profinite group and $p$ a prime number. We consider the finite field $\F_p$ as  a trivial discrete $G$-module. Let $\sC^\bullet=(C^\bullet(G,\F_p),\partial,\cup)$ be the differential graded algebra of inhomogeneous continuous cochains of $G$ with coefficients in $\F_p$ (see \cite[Ch.\ I, \S2]{NSW} and \cite[Section 3]{MT1}). We write $H^i(G,\F_p)$ for the corresponding cohomology groups. We denote by $Z^1(G,\F_p)$ the subgroup of $C^1(G,\F_p)$ consisting of all 1-cocycles. Because we use  trivial action on the coefficients $\F_p$, we have $Z^1(G,\F_p)=H^1(G,\F_p)={\rm Hom}(G,\F_p)$. Let $\chi_1,\chi_2,\chi_3$ be elements in $H^1(G,\F_p)$. Assume that 
\[
\chi_1\cup\chi_2=\chi_2\cup\chi_3=0\in H^2(G,\F_p).
\]
In this case we say that the triple Massey product $\langle \chi_1,\chi_2,\chi_3\rangle$ is defined. Then there exist cochains $D_{12}$ and $D_{23}$ in $C^1(G,\F_p)$ such that
\[
\partial D_{12}=\chi_1\cup \chi_2 \; \text{ and } \partial D_{23}=\chi_2\cup\chi_3
\]
in $C^2(G,\F_p)$. Then we say that $D:=\{\chi_1,\chi_2,\chi_3,D_{12},D_{23}\}$ (or sometimes for simplicity $\{D_{12},D_{23}\}$) is a defining system for the triple Massey product $\langle \chi_1,\chi_2,\chi_3\rangle$. Observe that 
\[
\partial (\chi_1\cup D_{23}+ D_{12}\cup\chi_3)=0,
\]
hence $\chi_1\cup D_{23}+D_{12}\cup\chi_3$ is a 2-cocycle. We define the value $\langle \chi_1,\chi_2,\chi_3\rangle_D$ of the triple Massey product $\langle \chi_1,\chi_2,\chi_3\rangle$ with respect to the defining system $D$ to be the cohomology class of $\chi_1\cup D_{23}+\chi_3\cup D_{12}$ in $H^2(G,\F_p)$. The set of all values $\langle \chi_1,\chi_2,\chi_3\rangle_D$ when $D$ runs over the set of all defining systems, is called the triple Massey product 
$\langle \chi_1,\chi_2,\chi_3\rangle \subseteq H^2(G,\F_p)$. If $0\in \langle \chi_1,\chi_2,\chi_3\rangle$, then we say that our triple Massey product vanishes.

\begin{defn}
 We say that $G$ has the {\it vanishing triple Massey product property (with respect to $\F_p$)} if every triple Massey product $\langle \chi_1,\chi_2,\chi_3\rangle$, where $ \chi_1,\chi_2, \chi_3\in H^1(G,\F_p)$, vanishes whenever it is defined.
\end{defn}

The Vanishing 3-Massey Conjecture then claims that for any field $F$ and any prime $p$, the absolute Galois group $G_F$ has the vanishing triple Massey product property. It was proved by  M. Hopkins and K. Wickelgren in \cite{HW} that if $F$ is a global field of characteristic not 2 and $p=2$, then $G_F$ has the vanishing triple Massey product property. In \cite{MT1} it was proved that the result of \cite{HW} is valid for any field $F$. 
In \cite{MT3} it was proved that $G_F$ has  the vanishing triple Massey product property with respect to $\F_p$ for any global field $F$ containing a primitive $p$-th root of unity.
In \cite{EMa}, I. Efrat and E. Matzri provided alternative proofs for the above mentioned  results in \cite{MT1} and \cite{MT3}.
In \cite{Ma}, E. Matzri  proved that for any prime $p$ and for any field $F$  containing a primitive $p$-th root of unity, $G_F$ has the vanishing triple Massey product property. In this paper we shall provide a cohomological proof to the main result in \cite{Ma} (see Theorem~\ref{thm:vanishing}). We also remove the assumption that $F$ contains a primitive  $p$-th root of unity (see Theorem~\ref{thm:main vanishing}). Thus every absolute Galois group has the vanishing triple Massey product property. This is a fundamental
new restriction on absolute Galois groups. See Subsection 4.3 for some significant consequences on the structure of quotients of absolute Galois groups. 
Indeed in Subsection 4.3 we are able to describe some strong restrictions on the shape of relations in Zassenhaus filtration $S_{(2)}$ modulo $S_{(4)}$ for any maximal pro-$p$-quotient of any absolute Galois group $G_F$. This shape of relations excludes the possibility of certain triple commutators occurring in these relations. These are new significant restrictions on the shape of relations in canonical quotients of absolute Galois groups. This description of relations in $S_{(2)}$ modulo $S_{(4)}$ appears to be close to an optimal description.

\raggedbottom
The structure of our paper is as follows. In Section 2 we recall basic material on the cohomology of bicyclic groups. In Section 3 we discuss  Heisenberg extensions. 
 In Section 4, by using the material developed in Sections 2 and 3, we provide an alternative cohomological proof of the main result of \cite{Ma} on the vanishing  of triple Massey products with respect to $\F_p$ when $F$ contains a primitive $p$-th root of unity. 
 Before we received a very nice preprint \cite{Ma} from E. Matzri, we planned to make such a proof but we completed this proof only after we received his preprint. We also want to notice that by referring directly to some results in \cite{Ti} (see Remark~\ref{rmk:Tignol} for a brief explanation), one can avoid using material in Section 2. 
 However we think that Section 2 might be of independent interest. E.Matzri used in his work tools from the theory of central simple algebras. 
 In our paper we use cohomological techniques instead. In Remark~\ref{rmk:modification} we provide yet another short direct variant of the key part of the proof of Theorem~\ref{thm:vanishing}.
Considering the details of our proof makes it possible to prove the vanishing of triple Massey products in a more general setting.
 Consider a formation $\{G,\{G(K)\},N \}$, where $G$ is a profinite group, $\{G(K)\}$ is a collection of open subgroups of $G$ indexed by a set $\Sigma=\{K\}$, and $N$ is a discrete $G$-module. 
 (See \cite[Chapter XI, \S 1]{Se} or \cite[Chapter VI, Section 6.1]{We} for a definition of a formation.)  
 We shall call such a formation $\{G,\{G(K)\},N \}$ a  {\it $p$-Kummer field formation} if it satisfies two axioms:
\begin{enumerate}
\item For each open normal subgroup $G(K)$ of $G$, $H^1(G/G(K),N^{G(K)})=\{1\}$. (Here $N^{G(K)}$ is the set of elements of $N$ which are fixed under the action of every element $\sigma$ in $G(K)$.)
\item There is a short exact sequence of $G$-module
\[
1\longrightarrow \Z/p\Z \longrightarrow N\stackrel{x\mapsto x^p}{\longrightarrow} N\longrightarrow 1,
\]
where $G$ acts trivially on $\Z/p\Z$, and $N$ is written in a multiplicative way.
\end{enumerate}

Recall that axiom (1) above guarantees that each $p$-Kummer field formation is also a field formation. (See \cite[Chapter 6, Section 6.2]{We}.) 
Our main interest is when $G=G_F$, the absolute Galois group of $F$, and $\Sigma$ is the set of all finite separable extensions of $F$, and $N=F_s^\times$. 
But our approach is valid in this more general setting which may have applications in anabelian geometry. Also this approach clarifies the key properties of $G$ which are sufficient for our proofs to go through.
Nearly simultaneously with our arXiv posting of the first version of our paper, I. Efrat and E. Matzri posted \cite{EMa2} on arXiv. The paper \cite{EMa2} is a replacement of \cite{Ma}. In \cite{EMa2}, I. Efrat and E. Matzri also provide  a cohomological approach to Theorem~\ref{thm:vanishing}. 
Their approach has a similar flavor to our proofs  in this paper, but it is still different. We feel that both papers taken together provide a definite complementary insight to the new fundamental property of absolute Galois groups. 
As we mentioned above, in Remark~\ref{rmk:modification} we provide the second alternative proof of the vanishing of  triple Massey products. In this proof we are able to show a specific element of the triple Massey product which vanishes. 
Some results of this paper (e.g. the results on Heisenberg extensions in Subsection 3.2, Lemma~\ref{lem:operators}) have already been used in the construction of important Galois groups. 
Namely in \cite{MT4} we succeeded in extending the crucial ideas in this paper together with further ideas in Galois theory to find explicit constructions of Galois extensions $L/F$ with ${\rm Gal}(L/F)\simeq \U_4(\F_p)$, for all fields $F$ and all primes $p$.
For example Theorem~\ref{thm:modification} and its proof play an important role in finding a crucial submodule of the $E^\times/(E^\times)^p$ where $E:=F(\sqrt[p]{a},\sqrt[p]{c})$, which is needed in our construction of a required Galois extension $L/F$ containing $E/F$ and having Galois group isomorphic to $\U_4(\F_p)$.
In Section 5 we prove the vanishing of $(k+1)$-fold Massey products of the form $\langle \chi_b,\chi_a,\ldots,\chi_a\rangle$ ($k$ copies of $\chi_a$)  and the vanishing of $(k+2)$-fold Massey products of the form $\langle\chi_a,\chi_b,\chi_a,\ldots,\chi_a\rangle$ ($k+1$ copies of $\chi_a$), where $k<p$. 
(See Theorem~\ref{thm:vanishing higher}.) The first vanishing can be deduced also from the results in \cite{Sha}, but the second vanishing appears to be new. (See also \cite[Corollary 3.17]{Wic} for another result on the vanishing of some $n$-fold Massey products.)
\\
\\
{\bf Acknowledgements:}  We are grateful to Jochen G\"{a}rtner and Adam Topaz with whom we initially began correspondence on the vanishing triple Massey products pursuing a different strategy. We are also grateful to Eliyahu Matzri for sending us his beautiful preprint \cite{Ma} shortly before his arXiv posting. Although we did not yet discuss details of our paper with Ido Efrat, Eliyahu Matzri, Danny Neftin and Kirsten Wickelgren, we thank them for their interest and great encouragement.  We are  grateful to an  anonymous referee for his/her  careful reading of our paper and for providing us with insightful comments and valuable suggestions which we used to improve our exposition.

\section{Cohomology of bicyclic groups}
In this section we study the cohomology of cyclic and bicyclic groups.  A number of basic results which we will need subsequently in this paper, are recalled here.  Our main references are \cite[ pages 16-19]{CKM} and \cite[pages 694-697]{Me}.
\subsection{Cohomology of cylic groups}
\label{subsec:cohomology of cyclic groups}
If $G$ is abelian group and $g\in G$ is an element of order n, we denote
\[
D_g:= g-1 \text{ and } N_g:= 1+g+\cdots+g^{n-1}\in \Z[G].
\]

Let $G$ be a cyclic group of order $n$. We choose a generator $s$ of $G$. Recall that $\epsilon \colon \Z[G]\to \Z$ is the augmentation homomorphism with $\epsilon(g)=1$ for all $g\in G$.
Then we have the following resolution of the trivial $G$-module $\Z$:
\[
\cdots \longrightarrow \Z[G] \stackrel{d_1}{\longrightarrow} \Z[G] \stackrel{d_0}{\longrightarrow} \Z[G] \stackrel{\epsilon}{\longrightarrow} \Z \longrightarrow 0,
\]
where $d_i$ is  multiplication by $D_s$ (resp. $N_s$) if $i$ is even (resp. odd).  For any $G$-module $M$, the above resolution determines  a  complex
\[
{\rm Hom}_G(\Z[G],M)\stackrel{d_0^*}{\longrightarrow} {\rm Hom}_G(\Z[G],M)\stackrel{d_1^*}{\longrightarrow} {\rm Hom}_G(\Z[G],M)\longrightarrow \cdots
\]
We make the natural identification ${\rm Hom}_G(\Z[G],M)=M$. Then the above complex becomes 
\[
M \stackrel{d_0^*}{\longrightarrow} M \stackrel{d_1^*}{\longrightarrow} M \stackrel{d_2^*}{\longrightarrow} M\longrightarrow \cdots
\]
This implies in particular that 
\[
H^{2}(G,M) \simeq \hat{H}^0(G,M):=\ker D_s/\im N_s = M^s/M_sA. 
\]
As explained in \cite[Chapter VIII, \S 4]{Se}, the above isomorphism does depend on the choice of generator $s$ and can be described as below. 
The choice of $s$ defines a homomorphism $\chi^s\colon G\to \Q/\Z$ such that $\chi^s(s)=1/n$. The  coboundary $\delta\colon H^1(G,\Q/\Z)\to H^2(G,\Z)$ associated to the short exact sequence of trivial $G$-modules
\[
0\to \Z\to \Q\to \Q/\Z\to 0,
\]
sends $\chi^s$ to an element $\theta_s=\delta \chi^s\in H^2(G,\Z)$. Then the isomorphism 
\[
M^s/N_sM \simeq H^2(G,M),
\]
is the map which sends $x\in M$ to $x\cup \delta\chi^s$.

\subsection{Cohomology of bicyclic groups}
\label{subsec:bicyclic}
Let $G$ be a bicyclic group. We choose two generators $s$, of order $m$, and $t$, of order $n$. 

We define a chain complex $L_\bullet=(L_i)$ as follows: $L_i=\Z[G]^{i+1}$ for all $i\geq 0$, and $d_i\colon L_{i+1}\to L_{i}$ are defined by the following conditions
\[
\begin{aligned}
&d_{2i} e_{2j}& =&N_s e_{2j-1}+D_t e_{2j},\\
&d_{2i}e_{2j+1}& = &D_s e_{2j}- N_t e_{2j+1},\\
&d_{2i+1} e_{2j} &=&N_s e_{2j-1}+N_t e_{2j},\\
&d_{2i+1}e_{2j+1}& =& D_s e_{2j}-D_t e_{2j+1},
\end{aligned}
\]
here, for convenience, we put $e_{-1}=0$ and $(e_0,\ldots,e_i)$ is the canonical basis of $L_i=\Z[G]^{i+1}$.
Then we obtain a free resolution of the trivial $G$-module $\Z$:
\[
\tag{1} \cdots\longrightarrow \Z[G]^4 \stackrel{d_2}{\longrightarrow} \Z[G]^3 \stackrel{d_1}{\longrightarrow} \Z[G]^2 \stackrel{d_0}{\longrightarrow} \Z[G] \stackrel{\epsilon}{\longrightarrow} \Z \longrightarrow 0.
\]

We define $I:=\ker\epsilon$, the augmentation ideal of $\Z[G]$; and $J:=\ker d_0$. Then we obtain the following exact sequence of $\Z[G]$-modules
\[
\tag{2} 0 \longrightarrow J\longrightarrow \Z[G]^2 \stackrel{f}{\longrightarrow} I\longrightarrow 0,
\]
where $f(x,y)=d_0(x,y)= D_t x+ D_sy$. We also consider the following exact sequence
\[
\tag{3} 0 \longrightarrow \Z[G]/\Z N_G\stackrel{g}{\longrightarrow} J \stackrel{h}{\longrightarrow} \Z^2\longrightarrow 0,
\]
where $N_G=\sum_{\sigma\in G} \sigma$, $g(x+\Z N_G)=(D_tx,-D_sx)$ and $h(x,y)=(\epsilon(x)/n,\epsilon(y)/m)$.

Now let $M$ be any $G$-module. The resolution (1) yields the following complex
\[
{\rm Hom}_G(\Z[G],M)\stackrel{d_0^*}{\to} {\rm Hom}_G(\Z[G]^2,M)\stackrel{d_1^*}{\to} {\rm Hom}_G(\Z[G]^3,M) \stackrel{d_2^*}{\to} {\rm Hom}_G(\Z[G]^4,M)\to \cdots
\]
We make the natural identifications ${\rm Hom}_G(\Z[G]^i,M)=M^i$. Then the above complex becomes 
\[
M \stackrel{d_0^*}{\longrightarrow} M^2 \stackrel{d_1^*}{\longrightarrow} M^3 \stackrel{d_2^*}{\longrightarrow} M^4\longrightarrow \cdots
\]
The explicit descriptions of some maps $d_i^*$ in matrix form are given below: 
\[
d_0^*=\begin{bmatrix}
D_t\\
D_s
\end{bmatrix},
d_1^*=\begin{bmatrix}
N_t&0\\
D_s &-D_t\\
0&N_s
\end{bmatrix},
d_2^*=\begin{bmatrix}
D_t& 0 & 0 \\
D_s &-N_t & 0\\
0 &N_s & D_t\\
0& 0 & D_s
\end{bmatrix}.
\]
In particular we have
\[
\begin{aligned}
Z^2(G,M):=&\ker d_2^*=\{(x,y,z)\in M^3\mid D_t(x)=0,D_s(z)=0, D_s(x)=N_t(y), D_t(z)=-N_s(z)\},\\
B^2(G,M):=&\im d_1^*=\{(x,y,z)\in A^3\mid  \exists (c,d)\in A^2: x=N_t(c),y=D_s(c)-D_t(d), z=N_s(d)\},\\
H^2(G,M)\simeq &Z^2(G,M)/B^2(G,M).
\end{aligned}
\]

The exact sequences (2) and (3) yield the following commutative diagram
\[
\tag{*}
\xymatrix{
&{\rm Hom}_G(\Z^2,M) \ar@{->}[d] \ar@{->}[rd]^{u}\\
{\rm Hom}_G(\Z[G]^2,M) \ar@{->}[r] \ar@{->}[rd]^{v}&{\rm Hom}_G(J,M) \ar@{->}[r] \ar@{->}[d] &{\rm Ext}_G^1(I,M) \ar@{->}[r] &0={\rm Ext}_G^1(\Z[G]^2,M) \\
&{\rm Hom}_G(\Z[G]/\Z N_G, M) \ar@{->}[d] \\
&{\rm Ext}_G^1(\Z^2,M)
}.
\]
This diagram implies that we have a natural injection 
\[
\eta\colon {\rm coker}(u)\hookrightarrow {\rm coker}(v).
\]
Note that $\eta$ is an isomorphism if and only if $H^1(G,M)=0$ since ${\rm Ext}_G^1(\Z^2,M)=H^1(G,M)^2$. Under the natural identifications 
\[{\rm Hom}_G(\Z^2,M)=(M^G)^2\; \text{ and } {\rm Hom}_G(\Z[G]^i,M)=M^i,\] 
we shall describe explicitly all objects and maps in the diagram.

First we have
\[
{\rm Hom}_G(\Z[G]/\Z N_G, M) =\{x\in M\mid N_Gx=0\}=:\,_{N_G}M.
\]
The map $v$ becomes $v \colon M^2\longrightarrow \,_{N_G}M$, which sends $(x,y)$ to $D_tx -D_sy$. This follows from the observation that the map $v$ is obtained by applying the functor ${\rm Hom}_{G}(\cdot, M)$  to the composite $\Z[G]/\Z N_G \to J\to \Z[G]^2$, which maps $x\mod \Z N_G$ to $(D_tx,-D_sx)$.

On the other hand, the surjection $d_1\colon \Z[G]^3\twoheadrightarrow J$ yields an injection ${\rm Hom}_G(J,M)\hookrightarrow {\rm Hom}_G(\Z[G]^3,M)=M^3$. If we identify ${\rm Hom}_G(J,M)$ with its image in $M^3$, then 
\[ {\rm Hom}_G(J,M)=\ker d_2^*=:Z^2(G,M), \text{ and } {\rm Ext}^1_G(I,M)=\ker d_2^*/\im d_1^*=H^2(G,M).\]
The map $u$ becomes $u\colon (M^G)^2\longrightarrow H^2(G,M)$, which can be described explicitly as follows. We consider the composite
\[
\varphi\colon \Z[G]^3\stackrel{d_1}{\longrightarrow}  J\stackrel{h}{\longrightarrow} \Z^2.
\]
Then 
\[
\begin{aligned}
\varphi(x,y,z)&= h(N_tx+D_sy,-D_ty+N_sz)\\
&=(\epsilon(N_tx+D_sy)/n, \epsilon(-D_ty+N_sz)/m)\\
&=(\epsilon(x),\epsilon(z)).
\end{aligned}
\]
The Hom-dual of $\varphi$ is the map 
\[
\varphi^*\colon (M^G)^2={\rm Hom}_G(\Z^2,M)\longrightarrow {\rm Hom}_G(J,M)\hookrightarrow {\rm Hom}_G(\Z[G]^3,M)=M^3,\] which is given by $\varphi^*(x,z)=(x,0,z)$. The map $u$ is then given by 
\[
u(x,z)=[(x,0,z)],
\]
where $[(x,0,z)]$ is the class of $(x,0,z)$ in $H^2(G,M)$. 

Let $\sigma\in G/\langle t\rangle$ denote the class of $s$ modulo $\langle t\rangle $. Then $G/\langle t\rangle=\langle \sigma\rangle$ is of order $m$. We have the natural identification $\hat{H}^0(G/\langle t\rangle,M^t)=M^G/N_\sigma(M^t)=H^2(G/\langle t\rangle,M)$ by identifying $a$ with the cup product $a\cup \theta_\sigma$. (See Subsection~\ref{subsec:cohomology of cyclic groups}.) 
Let $\chi^s\colon G\to \Q/\Z$ be a homomorphism such that $\chi^s(s)=1/n$, $\chi^s(t)=0$, and set $\theta_s:=\delta \chi^s\in H^2(G,\Z)$. 
Then we have
\[
[(0,0,z)]= z \cup\theta_s=z\cup \delta \chi^s\in H^2(G,M).
\]
This follows from an observation in \cite[page 18]{CKM} that the inflation map 
\[ {\rm inf}_{G/\langle t\rangle}\colon H^2(G/\langle t\rangle,M)\hookrightarrow H^2(G,M),\]
 is given by
\[
{\rm inf}_{G/\langle t\rangle}(z)=[(0,0,z)].
\]
(Note also that $\chi^s={\rm inf}_{G/\langle t\rangle}(\chi^\sigma)$ and $\theta_s={\rm inf}_{G/\langle t\rangle}\theta_\sigma$.) 

Similarly, we have
\[
[(x,0,0)]= x \cup\theta_t=x\cup \delta \chi^t\in H^2(G,M).
\]
Here $\chi^t\colon G\to \Q/\Z$ is a homomorphism such that $\chi^t(t)=1/n$, $\chi^t(s)=0$, and $\theta_t:=\delta \chi^t\in H^2(G,\Z)$. Therefore
\[
u(x,z)= x\cup\delta\chi^t+z\cup\delta\chi^s.
\]

Observe that we have the following commutative diagram
\[
\xymatrix{
\Z[G] \ar@{->}[r]^{x\mapsto (0,x,0)} \ar@{->>}[d]&\Z[G]^3\ar@{->>}[d]^{d_1}\\
\Z[G]/\Z[G] N_G \ar@{->}^-g[r] &J.
}
\]

This induces the following commutative diagram

\[ 
\xymatrix@C=2cm{
{\rm Hom}_G(\Z[G]^3,M)=M^3 \ar@{->}[r]^{(x,y,z)\mapsto y} &{\rm Hom}_G(\Z[G],M)=M\\
{\rm Hom}_G(J,M)=Z^2(G,M) \ar@{->}[r] \ar@{^{(}->}[u] &{\rm Hom}_G(\Z[G]/\Z N_G,M)=\,_{N_G}M \ar@{^{(}->}[u]
}
\]
In summary, under the identifications  ${\rm Hom}_G(\Z[G]^i,M)=M^i$,  ${\rm Hom}_G(J,M)=Z^2(G,M)$, ${\rm Ext}^1_G(I,M)=Z^2(G,M)/B^2(G,M)=H^2(G,M)$, etc., the diagram (*) becomes
\[
\tag{**}
\xymatrix{
&(M^G)^2 \ar@{->}[d] \ar@{->}[rd]^{u}\\
M^2 \ar@{->}[r] \ar@{->}[rd]^{v}&Z^2(G,M) \ar@{->}[r] \ar@{->}[d] &H^2(G,M) \ar@{->}[r] &0 \\
&\,_{N_G} M \ar@{->}[d] \\
&H^1(G,M)^2
}.
\]
Here  
\[ 
\begin{aligned}
v(x,y)&=D_tx -D_sy.\\
u(x,z)&= x\cup\delta\chi^t + z\cup\delta\chi^s.
\end{aligned}
\]
The natural injection  $\eta \colon {\rm coker}(u) \hookrightarrow {\rm coker}(v)$ is given by
\[
\eta([(x,y,z)])=[y].
\]
\section{Heisenberg extensions}
\subsection{Norm residue symbols}
Let $F$ be a field containing a primitive $p$-th root of unity $\xi$. 
For any element $a$ in $F^\times$, we shall write $\chi_a$  for the  character corresponding to $a$ via the Kummer map $F^\times\to H^1(G_F,\Z/p\Z)={\rm Hom}(G_F,\Z/pZ)$.  From now on we assume that $a$ is not in $(F^\times)^p$. The extension $F(\sqrt[p]{a})/F$ is a Galois extension with Galois group $\langle \sigma_a\rangle\simeq \Z/p\Z$, where $\sigma_a$ satisfies $\sigma_a(\sqrt[p]{a})=\xi\sqrt[p]{a}$. 

The character $\chi_a$ defines a homomorphism $\chi^a\in {\rm Hom}(G_F,\frac{1}p\Z/\Z)\subseteq {\rm Hom}(G_F,\Q/\Z)$ by the formula
\[
\chi^a =\frac{1}{p} \chi_a. 
\]
Let $b$ be any element in $F^\times$.  Then the norm residue symbol can be defined to be
\[
(a,b):= (\chi^a,b):= b\cup \delta \chi^a.
\]

The cup product $\chi_a\cup \chi_b\in H^2(G_F,\Z/p\Z)$ can be interpreted as the norm residue symbol $(a,b)$. More precisely, we consider the exact sequence
\[
0\longrightarrow \Z/p\Z \longrightarrow  F_s^\times \stackrel{x\mapsto x^p}{\longrightarrow} F_s^\times \longrightarrow 1,
\]
where $\Z/p\Z$ has been identified with the group of $p$-th roots of unity $\mu_p$ via the choice of $\xi$. As $H^1(G_F,F_s^\times)=0$, we obtain
\[
0{\longrightarrow} H^2(G_F,\Z/p\Z)\stackrel{i}{\longrightarrow} H^2(G_F,F_s^\times) \stackrel{\times p}{\longrightarrow} H^2(G_F,F_s^\times).
\]
Then one has $i(\chi_a\cup \chi_b)=(a,b)\in H^2(G_F,F_s^\times)$. (See \cite[Chapter XIV, Proposition 5]{Se}.)  The following fact (\cite[Chapter XIV, Proposition 2]{Se}) will also  be used frequently in the sequel.
\begin{prop}
\label{prop:killed by p-cyclic extension}
 We have
\[
\ker\left( H^2(G_F,F_s^\times) \stackrel{\res_{\ker\chi_a}}{\longrightarrow} H^2(G_{F(\sqrt[p]{a})},F_s^\times)\right)=\{(a,b)\mid b\in F^\times\}.
\]
\end{prop}

\subsection{Heisenberg extensions}
In this subsection we provide a short alternative version of some material in \cite[Section 5]{Ma}. (See also \cite[Chapter 2, Section 2.4]{Sha1}.)

Assume that $a,b$ are elements in $F^\times$, which are linearly independent modulo $(F^\times)^p$. Let $K= F(\sqrt[p]{a},\sqrt[p]{b})$. Then $K/F$ is a Galois extension whose Galois group is generated by $\sigma_a$ and $\sigma_b$. Here $\sigma_a(\sqrt[p]{b})=\sqrt[p]{b}$, $\sigma_a(\sqrt[p]{a})=\xi \sqrt[p]{a}$; $\sigma_b(\sqrt[p]{a})=\sqrt[p]{a}$, $\sigma_b(\sqrt[p]{b})=\xi \sqrt[p]{b}$. 

Let $\U_3(\Z/p\Z)$ be the group of all upper-triangular unipotent $3\times 3$-matrix with entries in $\Z/p\Z$. We consider a map $\U_3(\Z/p\Z)\to (\Z/p\Z)^2$ which sends $\begin{bmatrix} 
1 & x & z\\
0 & 1 & y\\
0 & 0 & 1 
\end{bmatrix}$
to $(x,y)$. Then we have the following embedding problem

\[
 \xymatrix{
& & &G_F \ar@{->}[d]^{\bar\rho} \\
0\ar[r]& \Z/p\Z \ar[r] &\U_3(\Z/p\Z)\ar[r] &(\Z/p\Z)^2\ar[r] &1,
}
\]
where $\bar\rho$ is the map $(\chi_a,\chi_b)\colon G_F\to {\rm Gal}(K/F)\simeq (\Z/p\Z)^2$. (The last isomorphism ${\rm Gal}(K/F)\simeq (\Z/p\Z)^2$ is the one which sends $\sigma_a$ to $(1,0)$ and $\sigma_b$ to $(0,1)$.)

Assume that $\chi_a\cup \chi_b=0$. Then the norm residue symbol $(a,b)$ is trivial. Hence there exists $\alpha$ in $F(\sqrt[p]{a})$ such that $N_{F(\sqrt[p]{a})/F}(\alpha)=b$ (see  \cite[Chapter XIV, Proposition 4 (iii)]{Se}). We set 
\[
A_0=\alpha^{p-1} \sigma_a(\alpha^{p-2})\cdots \sigma_a^{p-2}(\alpha)=\prod_{i=0}^{p-2} \sigma_a^{i}(\alpha^{p-i-1}) \in  F(\sqrt[p]{a}).
\]
\begin{lem} Let $f_a$ be an element in $F^\times$. Let $A=f_aA_0$. Then we have
\label{lem:operator}
\[
\frac{\sigma_a(A)}{A}=\frac{N_{F(\sqrt[p]{a})/F}(\alpha)}{\alpha^p}=\frac{b}{\alpha^p}.
\]
\end{lem}
\begin{proof}Observe that $\dfrac{\sigma_a(A)}{A}=\dfrac{\sigma_a(A_0)}{A_0}$.
The lemma then follows from the identity
\[
(s-1)\sum_{i=0}^{p-2} (p-i-1)s^{i} = \sum_{i=0}^{p-1} s^i -p s^0.
\qedhere
\]
\end{proof}

For any representation $\rho\colon G_F\to \U_3(\Z/p\Z)$ and $1\leq i< j\leq 3$, let $\rho_{ij}\colon G_F\to \Z/p\Z$ be the composition of $\rho$ with the projection from $\U_3(\Z/p\Z)$ to its $(i,j)$-coordinate.
\begin{prop}
\label{prop:Heisenberg extension} Assume that $\chi_a\cup \chi_b=0$. Let $f_a$ be an element in $F^\times$. Let  $A=f_aA_0$ be defined as above. Then the homomorphism $\bar{\rho}:= (\chi_a,\chi_b)\colon G_F\to \Z/p\Z\times \Z/p\Z$ lifts to a Heisenberg extension $\rho\colon G_F\to \U_3(\Z/p\Z)$ with ${\rm res}_{\ker\chi_a}(\rho_{13})=\chi_{A}$. 
\end{prop}
\begin{proof} 
From $\sigma_a(A)/A= b/\alpha^p\in (K^\times)^p$, and $\sigma_b(A)=A$, we see that $\sigma(A)/A \in (K^\times)^p$ for every $\sigma \in {\rm Gal}(K/F)$. This implies that the extension $L:=K(\sqrt[p]{A})/F$ is Galois. Let $\tilde{\sigma}_a\in {\rm Gal}(L/F)$ (resp. $\tilde{\sigma}_b\in {\rm Gal}(L/F)$) be an extension of $\sigma_a$ (resp. $\sigma_b$). Since $\sigma_b(A)=A$, we have $\tilde{\sigma}_b(\sqrt[p]{A})=\xi^j\sqrt[p]{A}$, for some $j\in \Z$. Hence $\tilde{\sigma}_b^p(\sqrt[p]{A})=\sqrt[p]{A}$. This implies that $\tilde{\sigma}_b$ is of order $p$.

On the other hand, we have
\[
\tilde{\sigma}_a(\sqrt[p]{A})^p=\sigma_a(A) =A \frac{b}{\alpha^p}.
\]
Hence $\tilde{\sigma}_a(\sqrt[p]{A})=\xi^i \sqrt[p]{A}\frac{\sqrt[p]{b}}{\alpha}$, for some $i\in \Z$. Then
\[
\tilde{\sigma}_a^p(\sqrt[p]{A})= \sqrt[p]{A} \frac{b}{N_{F(\sqrt[p]{a})/F}(\alpha)}=\sqrt[p]{A}.
\]
This implies that $\tilde{\sigma}_a$ is also of order $p$. We have
\[
\begin{aligned}
\tilde{\sigma}_a\tilde{\sigma}_b(\sqrt[p]{A}) &= \tilde{\sigma}_a(\xi^j\sqrt[p]{A})=\xi^{i+j}\sqrt[p]{A}\frac{\sqrt[p]{b}}{\alpha},\\
\tilde{\sigma}_b\tilde{\sigma}_a(\sqrt[p]{A}) &=\tilde{\sigma}_b(\xi^i \sqrt[p]{A}\frac{\sqrt[p]{b}}{\alpha})=\xi^{i+j}\sqrt[p]{A}\frac{\xi\sqrt[p]{b}}{\alpha}.
\end{aligned}
\]
We set $ \tilde{\sigma}_{A}:= \tilde{\sigma}_a \tilde{\sigma}_b\tilde{\sigma}_a^{-1}\tilde{\sigma}_b^{-1}$. Then
\[
\tilde{\sigma}_{A}(\sqrt[p]{A})=\xi\sqrt[p]{A}.
\]
This implies that $\tilde{\sigma}_{A}$ is of order $p$ and that ${\rm Gal}(L/F)$ is generated by $\tilde{\sigma}_a$ and $\tilde{\sigma}_b$. We also have
\[
\begin{aligned}
\tilde{\sigma}_a \tilde{\sigma}_{A}= \tilde{\sigma}_{A}\tilde{\sigma}_a, \;\text{ and } \tilde{\sigma}_b \tilde{\sigma}_{A}= \tilde{\sigma}_{A}\tilde{\sigma}_b.
\end{aligned}
\]
We can define an isomorphism $\varphi \colon {\rm Gal}(L/F)\to \U_3(\Z/p\Z)$ by letting
\[
\tilde{\sigma}_a \mapsto x:=\begin{bmatrix}
1& 1 & 0 \\
0& 1 & 0 \\
0& 0 & 1
\end{bmatrix},
\tilde{\sigma}_b\mapsto y:=
\begin{bmatrix}
1& 0 & 0 \\
0& 1 & 1 \\
0& 0 & 1 
\end{bmatrix},
\tilde{\sigma}_{A}\mapsto z:=
\begin{bmatrix}
1& 0 & 1 \\
0& 1 & 0 \\
0& 0 & 1 
\end{bmatrix}.
\]
Then the composition $\rho\colon G_F\to {\rm Gal}(L/F)\stackrel{\varphi}{\to} \U_3(\Z/p\Z)$ is the desired lifting of $\bar{\rho}$.
\end{proof}
\begin{cor}
\label{cor: Heisenberg} Let the notation be as in Proposition~\ref{prop:Heisenberg extension}. Let $\varphi_{ab}$ be the map $-\rho_{13}\colon G_F\to \Z/p\Z$. Then 
\[
d\varphi_{ab}= \chi_a\cup \chi_b \text { and } \res_{\ker\chi_a}(\varphi_{ab})=-\chi_A.
\]
\end{cor}
\begin{proof} Since $\rho\colon G_F\to \U_3(\Z/p\Z)$ is a homomorphism, we obtain
\[
d\varphi_{ab}(\sigma,\tau)= \rho_{13}(\sigma\tau)-\rho_{13}(\sigma)-\rho_{13}(\tau)=\rho_{12}(\sigma)\rho_{23}(\tau)=(\chi_a\cup\chi_b)(\sigma,\tau).
\]
Therefore $d\varphi_{ab}= \chi_a\cup \chi_b$, as desired.
\end{proof}
\section{Triple Massey products}
\subsection{Triple Massey products over fields containing primitive $p$-th roots of unity}
In this subsection we assume that $F$ is a field containing a primitive $p$-th root of unity.  Let $a$, $b$ and $c$ be elements in $F^\times$. Assume further that the triple Massey product $\langle \chi_a,\chi_b,\chi_c\rangle $ is defined, i.e., we have $\chi_a\cup\chi_b=0=\chi_b\cup \chi_c$.  Until Theorem~\ref{thm:vanishing}, we always assume further that $a$ and $c$ are linearly independent modulo $(F^\times)^p$, and that $a$ and $b$ are linearly independent modulo $(F^\times)^p$. We also fix two elements $f_a$ and $f_c$ in $F^\times$. Let $A_0$ be the element defined right before Lemma~\ref{lem:operator}. Let $A=f_aA_0$.

By Corollary~\ref{cor: Heisenberg}, there is a map $\varphi_{ab}\colon G_F\to \Z/p\Z$ such that
\[
d\varphi_{ab}= \chi_a\cup \chi_b \in C^2(G_F,\Z/p\Z) \text { and } \res_{\ker\chi_a}(\varphi_{ab})=-\chi_A.
\]
Since $\chi_b\cup\chi_c=0$ in $H^2(G_F,\Z/p\Z)$, there exists a map $\varphi_{bc}\colon G_F\to \Z/p\Z$ such that
\[
d\varphi_{bc}= \chi_b\cup \chi_c \in C^2(G_F,\Z/p\Z).
\]
Then 
\[
\langle \chi_a,\chi_b,\chi_c\rangle_\varphi := \chi_a\cup \varphi_{bc}+\varphi_{ab}\cup\chi_c 
\]
is an element in the triple Massey product $\langle \chi_a,\chi_b,\chi_c\rangle$.

We consider the following commutative diagram
\[
\xymatrix{
H^2(G_F,\Z/p\Z) \ar@{^{(}->}[r]^i \ar@{->}[d]^{\res_{\ker\chi_a}} & H^2(G_F,F_s^\times) \ar@{->}[d]^{\res_{\ker\chi_a}}\\
H^2(G_{F(\sqrt[p]{a})},\Z/p\Z) \ar@{^{(}->}[r]^i & H^2(G_{F(\sqrt[p]{a})},F_s^\times).
}
\]

\begin{lem}
\label{lem:restriction of Massey} We have

\begin{enumerate}
\item $
{\rm res}_{\ker \chi_a}(\langle \chi_a,\chi_b,\chi_c\rangle_\varphi)= \res_{\ker\chi_a}(\chi_c)\cup\chi_A \in H^2(G_{F(\sqrt[p]{a})},\Z/p\Z).
$
\item 
$
{\rm res}_{\ker \chi_a}(i(\langle \chi_a,\chi_b,\chi_c\rangle_\varphi))=  (c,A) \in H^2(G_{F(\sqrt[p]{a})},F_s^\times).
$
\end{enumerate}
\end{lem}
\begin{proof}
(1) We have
\[
\begin{aligned}
{\rm res}_{\ker \chi_a}(\langle \chi_a,\chi_b,\chi_c\rangle_\varphi)&= \res_{\ker\chi_a}(\chi_a\cup \varphi_{bc}+\varphi_{ab}\cup\chi_c )\\
&= \res_{\ker\chi_a}(\chi_a)\cup \res_{\ker\chi_a}(\varphi_{bc})+ \res_{\ker\chi_a}(\varphi_{ab})\cup\res_{\ker\chi_a}(\chi_c)\\
&= -\chi_A\cup \res_{\ker\chi_a}(\chi_c)\\
& = \res_{\ker\chi_a}(\chi_c)\cup \chi_A.
\end{aligned}
\]
(2) This follows from (1) and the commutativity of the above diagram.
\end{proof}

Let $E= F(\sqrt[p]{a},\sqrt[p]{c})$. Since $\chi_a\cup\chi_b=0=\chi_b\cup \chi_c$, we have $(a,b)=(b,c)=0$. Thus there are $\alpha$ in $F(\sqrt[p]{a})$ and $\gamma$ in $F(\sqrt[p]{c})$ such that
\[
N_{E/F(\sqrt[p]{a})}(\alpha)=b=N_{E/F(\sqrt[p]{c})}(\gamma).
\]

Let $G$ be the Galois group ${\rm Gal}(E/F)$. Then $G=\langle \sigma_a,\sigma_c \rangle $, where $\sigma_a\in G$ (respectively $\sigma_c\in G$) is an extension of $\sigma_a\in {\rm Gal}(F(\sqrt[p]{a})/F)$ (respectively $\sigma_c\in {\rm Gal}(F(\sqrt[p]{c})/F)$). We define 
\[
\begin{aligned}
C_0=\prod_{i=0}^{p-2} \sigma_c^{i}(\gamma^{p-i-1}) \in  F(\sqrt[p]{a}), 
\end{aligned}
\] 
$C:=f_cC_0$, and define $B:=\gamma/\alpha$. 
\begin{lem}
\label{lem:operators}
 We have
\begin{enumerate}
\item $\dfrac{\sigma_a(A)}{A}=N_{\sigma_c}(B)$.
\item $\dfrac{\sigma_c(C)}{C}=N_{\sigma_a}(B)^{-1}$.
\end{enumerate}
\end{lem}
\begin{proof}
(1) From Lemma~\ref{lem:operator}, we have
\[
\dfrac{\sigma_a(A)}{A}= \frac{b}{\alpha^p}=\frac{N_{\sigma_c}(\gamma)}{N_{\sigma_c}(\alpha)}=N_{\sigma_c}(B). 
\]
(2)
 From Lemma~\ref{lem:operator}, we have
\[
\dfrac{\sigma_c(C)}{C}= \frac{b}{\gamma^p}=\frac{N_{\sigma_a}(\alpha)}{N_{\sigma_a}(\gamma)}=N_{\sigma_a}(B)^{-1}. 
\]
\end{proof}

We consider $E^\times$ as a $G$-module under the Galois action. The diagram (**) in Subsection~\ref{subsec:bicyclic} becomes
\[
\tag{***}
\xymatrix{
&(F^\times)^2 \ar@{->}[d] \ar@{->}[rd]^{u}\\
(E^\times)^2 \ar@{->}[r] \ar@{->}[rd]^{v}&Z^2(G,E^\times) \ar@{->}[r] \ar@{->}[d] &H^2(G,E^\times) \ar@{->}[r] &0 \\
&\,_{N_G} E^\times \ar@{->}[d] \\
&0
}.
\]
Here  
\[ 
\begin{aligned}
v(x,y)&= \frac{\sigma_c(x)}{x}\frac{y}{\sigma_a(y)},\\
u(x,z)&= x\cup\delta\chi^{\sigma_c} + z\cup\delta\chi^{\sigma_a}.
\end{aligned}
\]
The natural isomorphism  $\eta \colon {\rm coker}(u) \hookrightarrow {\rm coker}(v)$ is given by
\[
\eta[(x,y,z)]=[y].
\]

In the following results in this subsection (except in Theorem~\ref{thm:vanishing}) we use the presentation of cohomology classes as in Section 2.
\begin{cor}
The  triple $(A,B,C)$ is  an element in $Z^2(G,E^\times)$. 
\end{cor}
\begin{proof} This follows immediately from the previous lemma.
\end{proof}

\begin{lem}
\label{lem: in image}
The element $[(A,B,C)]$ is in the image of $u$. 
\end{lem}

\begin{proof}
Since  $\eta$ is an isomorphism, it suffices to show that $\eta[(A,B,C)]$ is in the image of $v$.

We have 
\[ N_{\sigma_a\sigma_c}(B)=\frac{N_{\sigma_a\sigma_c}(\alpha)}{N_{\sigma_a\sigma_c}(\gamma)}=\frac{N_{\sigma_a}(\alpha)}{N_{\sigma_c}(\gamma)}=\frac{b}{b}=1.
\]
 Hence by Hilbert's Theorem 90, there exists $e \in E^\times$ such  that 
 \[B=\dfrac{\sigma_a\sigma_c(e)}{e} = \frac{\sigma_c(\sigma_a(e))}{\sigma_a(e)} \frac{e^{-1}}{\sigma_a(e^{-1})}.
 \]
Therefore $B$ is in $\im v$, as desired.
\end{proof}
\begin{cor}
\label{cor:decomposable} There exists $x,y\in F^\times$ such that
\[
{\rm inf}([(A,B,C)])= (a,x)+(c,y)\in H^2(G_F,F_s^\times).
\]
\end{cor}
\begin{proof}
Since $[(A,B,C)]$ is in the image of $u$, there exist $x$ and $y$ in $F^\times$ such that
\[
[(A,B,C)]= x\cup \delta \chi^{\sigma_a}+y\cup \chi^{\sigma_c}.
\]
This implies that
\[
\begin{aligned}
{\rm inf} ([(A,B,C)])&= x\cup \delta({\rm inf}(\chi^{\sigma_a})) +y\cup \delta({\rm inf}(\chi^{\sigma_c}))\\
&=x\cup\delta\chi^a+ y\cup \delta\chi^c\\
&= (x,a)+(y,b),
\end{aligned}
\]
as desired.
\end{proof}
\begin{rmk}
\label{rmk:Tignol} Observe that ${\rm coker}(v) = \hat{H}^{-1}(G,E^\times)$. We can identify naturally the group ${\rm coker}(v)$ with  group ${\rm U}_{\bold b}(G,E^\times)$ (the notation being as in \cite{Ti}). (See \cite[Remark after Corollary 1.6]{Ti}.) Then the composition map 
\[H^2(G,E^\times)\to{\rm coker}(u)\stackrel{\eta}{\to} {\rm coker}(v)={\rm U}_{\bold b}(G,E^\times),\]
 is exactly the map which was denoted by $\epsilon$ in \cite[page 423]{Ti}, $\epsilon\colon H^2(G,E^\times)\to {\rm U}_{\bold b}(G,E^\times)$.
  The proof of Lemma~\ref{lem: in image} shows that the element $[(A,B,C)]$ is in the kernel of $\epsilon$.

We have the following exact sequence
 \[
 0 \longrightarrow H^2(G,E^\times) \stackrel{{\rm inf}}{\longrightarrow} H^2(G_F,F_s^\times)\stackrel{{\rm res}}{\longrightarrow} H^2(G_E,F_s^\times). 
 \]
If we make the natural identification $H^2(G_F,F_s^\times)={\rm Br(F)}$ and $H^2(G_E,F_s^\times)={\rm Br}(E)$, then one can check the map 
\[ {\rm \inf}\colon H^2(G,E^\times) \stackrel{\simeq}{\longrightarrow} {\rm Br}(E/F):=\ker({\rm Br}(F)\to {\rm Br}(E)),\]
is  the natural isomorphism  $\pi \colon  H^2(G,E^\times) \longrightarrow {\rm Br}(E/F)$ mentioned in \cite[page 427]{Ti}. Then Corollary~\ref{cor:decomposable} follows from \cite[Proposition 1.5]{Ti}.
\QEDB
 \end{rmk}

We consider the following commutative diagram
\[
\xymatrix{
H^2(G,E^\times) \ar@{^{(}->}^{\rm inf}[r] \ar@{->}_{\res_{G/\langle\sigma_a\rangle}}[d] & H^2(G_F,F_s^\times) \ar@{->}^{\res}[r] \ar@{->}_{\res_{\ker\chi_a}}[d]&H^2(G_E,F_s^\times)\ar@{=}[d] \\
H^2(\langle\sigma_c\rangle,E^\times) \ar@{^{(}->}^{\rm inf}[r] & H^2(G_{F(\sqrt[p]{a})},F_s^\times) \ar@{->}^{\res}[r] &H^2(G_E,F_s^\times).
}
\]

\begin{lem}
\label{lem:restriction}
We have
\begin{enumerate}
\item $\res_{G/\langle \sigma_a\rangle}([A,B,C])=A\cup \delta\chi^{\sigma_c}\in H^2(\langle\sigma_c\rangle,E^\times)$.
\item $\res_{\ker\chi_a}({\rm inf}([A,B,C]))=(c,A)$ in $H^2(G_{F(\sqrt[p]{a})},F_s^\times)$.
\end{enumerate}
\end{lem}
\begin{proof}
(1) By \cite[page 17]{CKM} and under the identification $(E^\times)^{\sigma_c}/N_{\sigma_c}(E^\times)=H^2(\langle \sigma_c\rangle,E^\times)$ via identifying $x$ with the cup product $x\cup \delta(\chi^{\sigma_c})$, we have
\[
\res_{G/\langle \sigma_a\rangle}([A,B,C])= [A]=A\cup \delta\chi^{\sigma_c}. 
\]

(2) From the commutativity of the above diagram, we obtain
\[
\begin{aligned}
\res_{\ker\chi_a}({\rm inf}([A,B,C]))&= {\rm inf} (\res_{G/\langle \sigma_a\rangle}([A,B,C]))\\
&={\rm inf} (A\cup \delta(\chi^{\sigma_c}))\\
&=A\cup \delta({\rm inf}(\chi^{\sigma_c}))\\
&=A\cup \delta\chi^c\\
&=(c,A),
\end{aligned}
\]
as desired.
\end{proof}

\begin{cor}
\label{cor:relationship}
There exists $x\in F^\times$ such that 
\[
i(\langle \chi_a,\chi_b,\chi_c\rangle_{\varphi}) = {\rm inf}([A,B,C])+(a,x) \in H^2(G_F,F_s^\times).
\]
\end{cor}

\begin{proof}
From Lemma~\ref{lem:restriction of Massey} and Lemma~\ref{lem:restriction} we have
\[
\res_{\ker\chi_a}(i(\langle \chi_a,\chi_b,\chi_c\rangle_{\varphi}))= \res_{\ker\chi_a}({\rm inf}([A,B,C])) \in H^2(G_{F(\sqrt[p]{a}}),F_s^\times)).
\]
The statement then follows from Proposition~\ref{prop:killed by p-cyclic extension}.
\end{proof}
\begin{cor}
\label{cor:vanishing} There exist $x$ and $y$ in $F^\times$ such that
\[
\langle \chi_a,\chi_b,\chi_c\rangle_{\varphi}= \chi_a\cup\chi_x + \chi_c\cup\chi_y\in  H^2(G_F,\Z/p\Z).
\]
In particular, the triple Massey product $\langle \chi_a,\chi_b,\chi_c\rangle$ contains 0.
\end{cor}
\begin{proof}
By Corollaries~\ref{cor:decomposable} and \ref{cor:relationship}, there are $x_1,x_2$ and $y$ in $F^\times$ such that
\[
i(\langle \chi_a,\chi_b,\chi_c\rangle_{\varphi}) = {\rm inf}([A,B,C])+(a,x_2) = (a,x_1)+ (c,y)+(a,x_2) \in H^2(G_F,F_s^\times).
\]
Let $x:=x_1x_2$. Since $i\colon H^2(G_F,\Z/p\Z)\to H^2(G_F,F_s^\times)$ is injective, we obtain 
\[ \langle \chi_a,\chi_b,\chi_c\rangle_{\varphi}=\chi_a\cup\chi_x + \chi_c\cup \chi_y, \]
as desired.

For the last statement, we replace $\varphi_{ab}$ by $\varphi_{ab}^\prime=\varphi_{ab}+\chi_y$, and $\varphi_{bc}$ by $\varphi_{bc}^\prime=\varphi_{bc}-\chi_x$. (Here we consider $\chi_x$ and $\chi_y$ as  elements in $Z^1(G_F,\Z/p\Z)$.) Then $\varphi^\prime=\{ \varphi_{ab}^\prime,\varphi_{bc}^\prime\}$ is also a defining system for the triple Massey product $ \langle \chi_a,\chi_b,\chi_c\rangle$. Hence $ \langle \chi_a,\chi_b,\chi_c\rangle$ contains
\[  \langle \chi_a,\chi_b,\chi_c\rangle_{\varphi^\prime}=   \langle \chi_a,\chi_b,\chi_c\rangle_{\varphi}-\chi_a\cup \chi_x-\chi_c\cup\chi_y=0.
\qedhere
\]
\end{proof}

\begin{thm}
\label{thm:vanishing}
 Let $F$ be a field containing a primitive $p$-th root of unity. Let  $a$, $b$ and $c$ be elements in $F^\times$. The triple Massey product $\langle \chi_a,\chi_b,\chi_c\rangle$ contains 0 whenever it is defined.
\end{thm}
\begin{proof} Assume that $\langle \chi_a,\chi_b,\chi_c\rangle$ is defined. We can also assume that $a$, $b$ and $c$ are not in $(F^\times)^p$.
\\
\\
{\bf Case 1:} Assume that $a$ and $c$ are linearly dependent modulo $(F^\times)^p$. Let $\varphi=\{\varphi_{ab},\varphi_{bc}\}$ be a defining system for $\langle \chi_a,\chi_b,\chi_c\rangle$. We have
\[
\begin{aligned}
{\rm res}_{\ker \chi_a}(\langle \chi_a,\chi_b,\chi_c\rangle_\varphi)&= \res_{\ker\chi_a}(\chi_a\cup \varphi_{bc}+\varphi_{ab}\cup\chi_c )\\
&= \res_{\ker\chi_a}(\chi_a)\cup \res_{\ker\chi_a}(\varphi_{bc})+ \res_{\ker\chi_a}(\varphi_{ab})\cup\res_{\ker\chi_a}(\chi_c)\\
&= 0\cup \res_{\ker\chi_a}(\varphi_{bc})+ \res_{\ker\chi_a}(\varphi_{ab})\cup 0\\
& = 0.
\end{aligned}
\]
By Proposition~\ref{prop:killed by p-cyclic extension}, $\langle \chi_a,\chi_b,\chi_c\rangle_\varphi=\chi_a\cup\chi_x$ for some $x\in F^\times$. 
We replace $\varphi_{bc}$ by $\varphi_{bc}^\prime=\varphi_{bc}-\chi_x$. (Here we consider $\chi_x$ as an element in $Z^1(G_F,\Z/p\Z)$.) Then $\varphi^\prime=\{ \varphi_{ab},\varphi_{bc}^\prime\}$ is also a defining system for the triple Massey product $ \langle \chi_a,\chi_b,\chi_c\rangle$. Hence $ \langle \chi_a,\chi_b,\chi_c\rangle$ contains
\[  \langle \chi_a,\chi_b,\chi_c\rangle_{\varphi^\prime}=   \langle \chi_a,\chi_b,\chi_c\rangle_{\varphi}-\chi_a\cup \chi_x=0.
\]
 \\
\\
{\bf Case 2:} Assume that $a$ and $b$ are linearly dependent modulo $(F^\times)^p$. Then $\chi_b=\lambda \chi_a$, for some $\lambda\in \Z/p\Z$. 
By the linearity of Massey products, we have
\[
\lambda \langle \chi_a,\chi_a,\chi_c\rangle\subseteq\langle \chi_a,\lambda\chi_a,\chi_c\rangle=\langle \chi_a,\chi_b,\chi_c\rangle.
\]
(See \cite[Lemma 6.2.4 (ii)]{Fe}.) So it is enough to show that $\langle \chi_a,\chi_a,\chi_c\rangle$ contains 0. 
If $p=2$ then $\langle \chi_a,\chi_b,\chi_c\rangle$ contains 0 by \cite[Theorem 1.2]{MT1}. So we may and shall assume that $p>2$. Since then our coefficients are $\F_p$ of characteristic not 2, we can divide by $2$ and define
\[
\varphi_{aa}:=- \frac{\chi_a^2}{2}.
\]
Then it is straightforward to check that
\[
\begin{aligned}
d\varphi_{aa}(\sigma,\tau)= \varphi_{aa}(\sigma)+\varphi_{aa}(\tau)-\varphi_{aa}(\sigma\tau)=\chi_a(\sigma)\chi_a(\tau)= \chi_a\cup\chi_a(\sigma,\tau).
\end{aligned}
\]
We pick any map $\varphi_{ac}\colon G_F\to \Z/p\Z$ such that $d\varphi_{ac}=\chi_a\cup\chi_c$. Then $\{\varphi_{aa},\varphi_{ac}\}$ is a defining system for $\langle \chi_a,\chi_a,\chi_c\rangle$. We have
\[
\begin{aligned}
{\rm res}_{\ker \chi_a}(\langle \chi_a,\chi_a,\chi_c\rangle_\varphi)&= \res_{\ker\chi_a}(\chi_a\cup \varphi_{ac}+\varphi_{aa}\cup\chi_c )\\
&= \res_{\ker\chi_a}(\chi_a)\cup \res_{\ker\chi_a}(\varphi_{ac})+ \res_{\ker\chi_a}(\varphi_{aa})\cup\res_{\ker\chi_a}(\chi_c)\\
&= 0\cup \res_{\ker\chi_a}(\varphi_{ac})+ 0\cup\res_{\ker\chi_a}(\chi_c) \\
& = 0.
\end{aligned}
\]
By Proposition~\ref{prop:killed by p-cyclic extension}, $\langle \chi_a,\chi_a,\chi_c\rangle_\varphi=\chi_a\cup\chi_x$ for some $x\in F^\times$. Use the same argument as in Case 1, this implies that $\langle \chi_a,\chi_a,\chi_c\rangle$ contains 0. (This also follows from a more general result \cite[Proposition 4.2]{Sha}. See also Theorem~\ref{thm:vanishing higher}.)
\\
\\
{\bf Case 3:} Assume that $a$ and $c$ are linearly independent and that $a$ and $b$ are linearly independent. Then Corollary~\ref{cor:vanishing} says that  the triple Massey product $\langle \chi_a,\chi_b,\chi_c\rangle$ contains 0.
\end{proof}
\begin{rmk}
\label{rmk:modification}
 In Case 3 in the above proof, we can show that  a specific element of  the Massey triple product vanishes. This leads to another proof for this case and hence for Theorem~\ref{thm:vanishing}, which avoids using Corollary~\ref{cor:decomposable}.  

\begin{thm}
\label{thm:modification}
 There exists $f_a,f_c$ in $F^\times$ such that $(f_aA_0,B,f_cC_0)$ is a coboundary. 
\end{thm}
\begin{proof}
As in the proof of Lemma~\ref{lem: in image}, there exists $e\in E^\times$ such that $B= \dfrac{\sigma_a\sigma_c(e)}{e}$.

By Lemma~\ref{lem:operators}, we have
\[
\begin{aligned}
\frac{\sigma_a(A_0)}{A_0}=N_{\sigma_c}(B)= N_{\sigma_c}\left(\frac{\sigma_a\sigma_c(e)}{e}\right)=\frac{\sigma_a(N_{\sigma_c}(e))}{N_{\sigma_c}(e)}.
\end{aligned}
\]
This implies that
\[
\frac{N_{\sigma_c}(e)}{A_0}=\sigma_a\left(\frac{N\sigma_c(e)}{A_0}\right).
\]
Hence 
\[
\dfrac{N_{\sigma_c}(e)}{A_0} \in F(\sqrt[p]{c})^\times\cap F(\sqrt[p]{a})^\times=F^\times.
\]
Therefore, there exists $f_a\in F^\times$ such that
$
N_{\sigma_c}(e)= A_0 f_a.
$

Similarly, by Lemma~\ref{lem:operators}, we have
\[
\begin{aligned}
\frac{\sigma_c(C_0)}{C_0}=N_{\sigma_a}(B^{-1})= N_{\sigma_a}\left(\frac{\sigma_a\sigma_c(e^{-1})}{e^{-1}}\right)=\frac{\sigma_c(N_{\sigma_a}(e^{-1}))}{N_{\sigma_a}(e^{-1})}.
\end{aligned}
\]
This implies that
\[
\frac{N_{\sigma_a}(e^{-1})}{C_0}=\sigma_c\left(\frac{N\sigma_a(e^{-1})}{C_0}\right).
\]
Hence 
\[
\dfrac{N_{\sigma_a}(e^{-1})}{C_0} \in F(\sqrt[p]{a})^\times\cap F(\sqrt[p]{c})^\times=F^\times.
\]
Therefore, there exists $f_c\in F^\times$ such that
$
N_{\sigma_a}(e^{-1})= C_0 f_c.
$

Let $C_1=\sigma_c(e)$ and $C_2=e^{-1}$. Then we have
\[
\begin{aligned}
B&=\dfrac{\sigma_a\sigma_c(e)}{e} = \frac{\sigma_a(\sigma_c(e))}{\sigma_c(e)} \frac{e^{-1}}{\sigma_c(e^{-1})}=\frac{\sigma_a(C_1)}{C_1} \frac{C_2}{\sigma_c(C_2)}, \\
N_{\sigma_c}(C_1)&=N_{\sigma_c}(\sigma_c(e))=N_{\sigma_c}(e)=A_0f_a,\\
N_{\sigma_a}(C_2)&=N_{\sigma_a}(e^{-1})=C_0 f_c.
\end{aligned}
\]
This implies that $(A_0f_a,B,C_0f_c)$ is in $B^2(G,E^\times)$.
\end{proof}
As a consequence of Theorem~\ref{thm:modification}, we obtain another proof for Case 3 in the proof of Theorem~\ref{thm:vanishing}.
\begin{proof}[Another proof for Case 3 in Proof of Theorem~\ref{thm:vanishing}]
Let $A=f_aA_0$ and $C=f_cC_0$ so that we use the same notation as  before. We replace $\varphi_{bc}$ by $\varphi_{bc}^\prime=\varphi_{bc}-\chi_x$. Then $\varphi^\prime=\{\varphi_{ab},\varphi_{bc}^\prime\}$ is also a defining system for the triple Massey product $\langle\chi_a,\chi_b,\chi_c\rangle$. 
 Corollary~\ref{cor:relationship} implies that $i(\langle\chi_a,\chi_b,\chi_c\rangle)$ contains
 \[
 \begin{aligned}
 i(\langle\chi_a,\chi_b,\chi_c\rangle_{\varphi^\prime})&=i(\langle\chi_a,\chi_b,\chi_c\rangle_{\varphi})-i(\chi_a\cup\chi_x)\\
 &= {\rm inf}([A,B,C])+(a,x)-(a,x)= {\rm inf}([A,B,C]).
 \end{aligned}
 \]
Hence $\langle\chi_a,\chi_b,\chi_c\rangle$ contains $i^{-1}({\rm inf}([A,B,C]))=0$.
\end{proof}
\end{rmk}
\subsection{Embedding problems and triple Massey products over an arbitrary  field}
A {\it weak embedding problem} $\sE$ for a profinite group $\Pi$ is a diagram 
\[
\sE:=
\xymatrix
{
{} & \Pi \ar[d]^{\alpha}\\
U \ar[r]^f & \bar U
}
\] 
which consists of  profinite groups $U$ and $\bar U$ and homomorphisms $\alpha \colon \Pi\to \bar U$, $f\colon U\to \bar U$ with $f$ being surjective. (All homomorphisms of profinite groups considered in this paper are assumed to be continuous.) If in addition $\alpha$ is also surjective, we call $\sE$ an {\it embedding problem}. 

A {\it weak solution} of $\sE$ is  a homomorphism $\beta\colon \Pi\to U$ such that $f\beta=\alpha$. 
 We call $\sE$ a {\it finite} weak embedding problem if  group $U$ is finite. The {\it kernel} of $\sE$ is defined to be $N:=\ker(f)$. 

Let $\U_4(\F_p)$ be the group of all upper-triangular unipotent $4\times 4$-matrices with entries in  $\F_p$. Let $Z$ be the subgroup of all such matrices with all off-diagonal entries being $0$ except at position $(1,4)$. We may identify $\U_4(\F_p)/Z$ with the group $\bar\U_4(\F_p)$ of all upper-triangular unipotent $4\times 4$-matrices with entries over $\F_p$ with the $(1,4)$-entry omitted.
For any representation $\rho\colon G\to \U_4(\F_p)$ and $1\leq i< j\leq 4$, let $\rho_{ij}\colon G\to \F_p$ be the composition of $\rho$ with the projection from $\U_4(\F_p)$ to its $(i,j)$-coordinate. We use  similar notation for representations $\bar\rho\colon G\to\bar\U_4(\F_p)$. Note that for each $i=1,2,3$, $\rho_{i,i+1}$ (resp., $\bar\rho_{i,i+1}$) is a group homomorphism.

Recall that we have the following result (\cite[Lemma 3.1]{MT2}), which is a direct consequence of \cite[Theorem 2.4]{Dwy}. 

\begin{lem}
\label{lem:vanishingEP}
Let $G$ be a profinite group, and $p$ a prime number. Then the following statements are equivalent:
\begin{enumerate}
\item $G$ has the vanishing triple Massey product property with respect to $\F_p$.
\item For every homomorphism $\bar\rho\colon G\to \bar\U_4(\F_p)$, the finite weak embedding problem 
\[
\xymatrix{
& & &G \ar@{->}[d]^{(\bar\rho_{12},\bar\rho_{23},\bar\rho_{34})} \ar@{-->}[ld]\\
0\ar[r]& M \ar[r] &\U_4(\F_p)\ar[r] &(\F_p)^3\ar[r] &1,
}
\]
has a weak  solution, i.e., $(\bar\rho_{12},\bar\rho_{23},\bar\rho_{34})$ can be lifted to a homomorphism $\rho\colon G\to \U_4(\F_p)$.
\end{enumerate}
\end{lem}
\begin{prop}
\label{prop: reduction coprime p}
 Let $G$ be a profinite group, and $p$ a prime number. Let $H$ be an open subgroup of $G$ whose index is coprime to $p$. Assume that $H$ has the vanishing triple Massey product property with respect to $\F_p$, then $G$ also has the vanishing triple Massey product property with respect to $\F_p$.
\end{prop}
\begin{proof}
We shall prove the condition (2) in Lemma \ref{lem:vanishingEP} for the group $G$. 

Let $\bar\rho\colon G\to \bar\U_4(\F_p)$ be any homomorphism. We   consider the weak embedding problem
\[
(\sE) \xymatrix{
& & &G \ar@{->}[d]^{(\bar\rho_{12}, \bar\rho_{23}, \bar\rho_{34})=:\phi} \\
0\ar[r]& M \ar[r] &\U_4(\F_p)\ar[r] &(\F_p)^3\ar[r] &1.
}
\]

Then by the assumption and by Lemma \ref{lem:vanishingEP},   the weak embedding problem $(\sE\mid_H)$ 
\[
(\sE\mid_H) \xymatrix{
& & &H \ar@{->}[d]^{\phi\mid_H} \ar@{-->}[ld]\\
0\ar[r]& M \ar[r] &\U_4(\F_p)\ar[r] &(\F_p)^3\ar[r] &1,
}
\]
which is induced from $(\sE)$, has a weak solution.  Let $\epsilon$ be the cohomology class in $H^2((\F_p)^3,M)$ which corresponds to the extension
\[
1 \longrightarrow M \longrightarrow \U_4(\F_p) \xrightarrow{(a_{12},a_{23},a_{34})} (\F_p)^3\longrightarrow 1.
\]
We have the following commutative diagram
\[
\xymatrix{
H^2((\F_p)^3,M) \ar@{=}[d] \ar@{->}[r]^{\phi^*} & H^2(G,M) \ar@{->}[d]_{{\rm res}}\\
H^2((\F_p)^3,M) \ar@{->}[r]^-{(\phi\mid_H)^*} & H^2(H,M).
}
\]
In particular, $({\phi\mid}_H)^*(\epsilon)={\rm res}(\phi^*(\epsilon))$. Since the weak embedding problem $({\sE\mid}_H)$ has a weak solution, we see that $({\phi\mid}_H)^*(\epsilon)=0$ by Hoechsmann's lemma (\cite[Chapter 3, \S 5, Proposition 3.5.9]{NSW}). (Note that the statement of Hoechsmann's lemma in \cite{NSW} deals with embedding problems, but its proof goes well with  weak embedding problems.)
  Since $[G:H]$ is coprime to $p$ and the order of $M$ is a $p$-power, we see that the restriction map ${\rm res}\colon H^2(G,M)\longrightarrow H^2(H,M)$ is injective by \cite[Chapter I, \S 2, Corollary to Proposition 9]{Se2}. Hence $\phi^*(\epsilon)=0$. Hoechsmann's lemma implies that the weak embedding problem $(\sE)$ has a weak solution, and we are done.
\end{proof}

\begin{thm}
\label{thm:main vanishing}
 Let $F$ be any field and $p$ a prime number. Then the absolute Galois group $G_F$ of $F$ has the vanishing triple Massey product property with respect to $\F_p$
\end{thm}
\begin{proof}
If ${\rm char} F=p$, then the maximal pro-$p$-quotient $G_F(p)$ of $G_F$ is a free pro-$p$-group. Therefore $G_F(p)$ and $G_F$ has the vanishing triple Massey product property with respect to $\F_p$.

Now we assume that ${\rm char} F\not=p$. Let $\xi$ be a primitive $p$-th root of unity, and let $K=F(\xi)$. Then $K/F$ is a finite extension of degree $d$, and $d$ divides $p-1$. This implies that $[G_F:G_K]=d$ is coprime to $p$. Since $G_K$ has the vanishing triple Massey product property by Theorem~\ref{thm:vanishing}, it follows that $G_F$ also has the vanishing triple Massey product property by Proposition~\ref{prop: reduction coprime p}. 
\end{proof}

\subsection{Some consequences} In this subsection we assume that $p$ is an odd prime number. The case when $p= 2$ is treated in \cite[Theorem 1.3 and Theorem 1.4]{MT1}.

Recall that for a profinite group $G$ and a prime number $p$,  the Zassenhaus ($p$-)filtration $(G_{(n)})$ of $G$ is defined inductively by
\[
G_{(1)}=G, \quad G_{(n)}=G_{(\lceil n/p\rceil)}^p\prod_{i+j=n}[G_{(i)},G_{(j)}],
\]
where $\lceil n/p \rceil$ is the least integer which is greater than or equal to $n/p$. (Here for two closed subgroups $H$ and $K$ of $G$,  $[H,K]$ means the smallest closed subgroup of $G$ containing the  commutators $[x,y]=x^{-1}y^{-1}xy$, $x\in H, y\in K$. Similarly, $H^p$ means the smallest closed subgroup of $G$ containing  the $p$-th powers $x^p$, $x\in H$.)

Let $(I,<)$ be a well-ordered set. Let $S$ be a free pro-$p$-group on  a set of generators $x_i, i\in I$ (see \cite[Definition 3.5.14]{NSW}). Let $S_{(i)}$, $i=1,2,\ldots$ be the $p$-Zassenhaus filtration of $S$. 
Then any element $r$ in $S_{(2)}$ may be written uniquely as
\begin{equation}
r=
\begin{cases}
\displaystyle  \prod_{i<j}[x_i,x_j]^{b_{ij}}\prod_{i\in I} x_i^{3a_i}  \prod_{i< j, k\leq j}[[x_i,x_j],x_k]]^{c_{ijk}}\cdot r', \text{ if } p=3,\\
\displaystyle  \prod_{i<j}[x_i,x_j]^{b_{ij}}\prod_{ i< j, k\leq j}[[x_i,x_j],x_k]]^{c_{ijk}}\cdot r', \text { if } p\not=3,
\end{cases}
\label{modulo S4}
\end{equation}
where $a_i,b_{ij},c_{ijk}\in \{0,1,\ldots,p-1\}$ and $r^\prime\in S_{(4)}$. For convenience we call (\ref{modulo S4}) the canonical decomposition modulo $S_{(4)}$ of $r$ (with respect to the basis $(x_i)$) and we also set $u_{ij}=b_{ij}$ if $i<j$, and  $u_{ij}=b_{ji}$ if $j<i$.

We denote by $G_F(p)$  the maximal pro-$p$-quotient of an absolute Galois group $G_F$ of a given field $F$.

\begin{thm}
\label{thm:main ob1}
 Let $\sR$ be a set of elements in $S_{(2)}$. Assume that there exists an element $r$ in $\sR$ and distinct indices $i,j,k$ with $i<j, k<j$ such that: 
\begin{enumerate}
\item[(i)]  In (\ref{modulo S4})  the canonical decomposition modulo $S_{(4)}$ of $r$, 
$u_{ij}=u_{kj}=u_{ki}=u_{kl}=u_{jl}=0$  for all $l\neq i,j,k$, and  $c_{ijk}\not=0$; and
\item[(ii)]  for every $s\in \sR$ which is different from  $r$,  the factors $[x_k,x_i]$, $[x_i,x_k]$ and $[x_i,x_j]$ do not occur in the canonical decomposition modulo $S_{(4)}$ of $s$. 
\end{enumerate}
Then $G=S/\langle \sR \rangle$ is not realizable as $G_F(p)$ for any field $F$.
\end{thm}
\begin{proof} This follows immediately from Theorem~\ref{thm:main vanishing}, \cite[Theorem 7.8]{MT1} and \cite[Corollary 3.5]{MT1}.
\end{proof}

\begin{thm}
\label{thm:main ob2}
Let $\sR$ be a set of elements in $S_{(2)}$. Assume that there exists an element $r$ in $\sR$ and distinct indices $i<j$ such that: 
\begin{enumerate}
\item[(i)]  In (\ref{modulo S4}) the canonical decomposition modulo $S_{(4)}$ of $r$,  
$u_{ij}=u_{il}=u_{jl}=0$, for all $l\not=i,j$ and $c_{iji}\not=0$ (respectively, $c_{ijj}\not=0$); and
\item[(ii)] for every $s\in \sR$ which is different from  $r$,  the factor $[x_i,x_j]$  does not occur in the canonical decomposition modulo $S_{(4)}$ of $s$. 
\end{enumerate}
Then $G=S/\langle \sR \rangle$ is not realizable as $G_F(p)$ for any field $F$.
\end{thm}
\begin{proof} This follows immediately from Theorem~\ref{thm:main vanishing}, \cite[Theorem 7.12]{MT1} and \cite[Corollary 3.5]{MT1}.
\end{proof}

\section{Vanishing of some higher Massey products}
\label{sec: higher Massey}
\subsection{Massey products}
Let $G$ be a profinite group and $p$ a prime number. We consider the finite field $\F_p$ as  a trivial discrete $G$-module. Let $\sC^\bullet=(C^\bullet(G,\F_p),\partial,\cup)$ be the differential graded algebra of inhomogeneous continuous cochains of $G$ with coefficients in $\F_p$ \cite[Ch.\ I, \S2]{NSW}. We write $H^i(G,\F_p)$ for the corresponding cohomology groups. We denote by $Z^1(G,\F_p)$ the subgroup of $C^1(G,\F_p)$ consisting of all 1-cocycles. Because we use  trivial action on the coefficients $\F_p$, $Z^1(G,\F_p)=H^1(G,\F_p)={\rm Hom}(G,\F_p)$.
 (See \cite{MT1,MT2} and references therein for more general setups.)

Let $n\geq 3$ be an integer. Let $a_1,\ldots,a_n$ be elements in $H^1(G,\F_p)=Z^1(G,\F_p)\subseteq C^1(G,\F_p)$.
\begin{defn}
\label{defn:Massey product}
 A collection $\sM=\{a_{ij}\mid 1\leq i<j\leq n+1, (i,j)\not=(1,n+1)\}$ of elements $a_{ij}$ of $\sC^1(G,\F_p)$ is called a {\it defining system} for the {\it $n$-fold Massey product} $\langle a_1,\ldots,a_n\rangle$ if the following conditions are fulfilled:
\begin{enumerate}
\item $a_{i,i+1}=a_i$ for all $i=1,2\ldots,n$.
\item $\partial a_{ij}= \sum_{l=i+1}^{j-1} a_{il}\cup a_{lj}$ for  all $i+1<j$.
\end{enumerate}
Then $\sum_{k=2}^{n} a_{1k}\cup a_{k,n+1}$ is a $2$-cocycle.
Its  cohomology class in $H^2$  is called the {\it value} of the product relative to the defining system $\sM$,
and is denoted by $\langle a_1,\ldots,a_n\rangle_\sM$.
The {\it $n$-fold Massey product} $\langle a_1,\ldots, a_n\rangle$ itself is the subset of $H^2(G,\F_p)$ consisting of all  elements which can be written in the form $\langle a_1,\ldots, a_n\rangle_\sM$ for some defining system $\sM$.

When $n=3$ we will speak about a {\it triple} Massey product. Note that in this case the triple Massey  product $\langle a_1,a_2,a_3\rangle$ is defined if and only if $a_1\cup a_2=a_2\cup a_3=0$ in $H^2(G,\F_p)$.
\end{defn}

For some convenience, we introduce the following definition.
\begin{defn}
Let $n\geq 1$ be an integer. Let $a_1,\ldots,a_n$ be elements in $H^1(G,\F_p)$.
A collection $\sM=\{a_{ij}\mid 1\leq i<j\leq n+1\}$ of elements $a_{ij}$ of $\sC^1(G,\F_p)$ is called a {\it complete defining system} for the $n$-tuple $(a_1,\ldots,a_n)$ if the following conditions are fulfilled:
\begin{enumerate}
\item $a_{i,i+1}=a_i$ for all $i=1,2\ldots,n$.
\item $\partial a_{ij}= \sum_{l=i+1}^{j-1} a_{il}\cup a_{lj}$ for  all $i+1<j$.
\end{enumerate}
\end{defn}
Note that for $n=1$, $\sM=\{a_{12}:=a_1\}$ is a complete defining system for $(a_1)$. For $n=2$, $(a_1,a_2)$ has a complete defining system if and only if $a_1\cup a_2=0$. For $n\geq 3$, $(a_1,\ldots,a_n)$ has a complete defining system if and only if the $n$-fold Massey product $\langle a_1,\ldots,a_n\rangle$ is defined and contains 0.
\subsection{Vanishing of some higher Massey products}
Let $A$ be a unital commutative ring. Let $n$ be a positive integer. Assume that every integer $1\leq k\leq n$ is invertible in $A$. We have the following binomial polynomials in the ring $A[X]$ of polynomials  in one variable $X$ and with coefficients in $A$:
\[
{X \choose 0}= 1 \text{ and } {X\choose k}= \dfrac{X(X-1)\cdots (X-k+1)}{k!}, 1\leq k\leq n.
\]
We have the following elementary lemma. 
\begin{lem}
\label{lem:binom} Let the notation be as above. 
 For $1\leq k\leq n$, one has
\[ 
{X+Y\choose k} =\sum_{l=0}^k {X\choose l}{Y\choose k-l} \in A[X,Y].
\]
Here $A[X,Y]$ is the ring of polynomials with coefficients $A$ in two variables $X$ and $Y$.
\qed
\end{lem}
\begin{rmk} Let the notation be as above. Then for $1\leq k\leq n$, one also has
\[
\frac{(X+Y)^k}{k!} =\sum_{l=0}^k \frac{X^l}{l!}\frac{Y^{k-l}}{(k-l)!} \in A[X,Y].
\]
And all results presented below can be stated and proved by using $\dfrac{X^k}{k!}$ (with some obvious modifications) instead of using $\displaystyle {X \choose k}$.
\QEDB
\end{rmk}

For any element $\chi$ in $H^1(G,\F_p)=Z^1(G,\F_p)={\rm Hom}(G,\F_p)$ and for each $k=0,\ldots,p-1$, we define $\displaystyle {\chi \choose k} \in C^1(G,\F_p)$ by
\[
{\chi \choose k}(\sigma)= {\chi(\sigma) \choose k},\; \forall \sigma\in G.
\]
\begin{cor}
\label{cor:d of binom}
Let $\chi$ be an element in ${\rm Hom}(G,\F_p)$. Let $k$ be an integer with $1\leq k<p$. Then
\[
d {\chi\choose k} = -\sum_{l=1}^{k-1} {\chi \choose l}\cup {\chi \choose k-l} \in C^2(G,\F_p).
\]
\end{cor}

\begin{proof} By Lemma~\ref{lem:binom}, we have
\[
\begin{aligned}
d {\chi\choose k}(\sigma,\tau)&= {\chi\choose k}(\sigma) + {\chi\choose k}(\tau) - {\chi\choose k}(\sigma\tau)\\
&= {\chi(\sigma)\choose k} + {\chi(\tau) \choose k} - {\chi(\sigma\tau)\choose k}\\
&= {\chi(\sigma)\choose k} + {\chi(\tau) \choose k} - {\chi(\sigma)+\chi(\tau)\choose k}\\
&= {\chi(\sigma)\choose k} + {\chi(\tau) \choose k} - \sum_{l=0}^{k} {\chi(\sigma) \choose l}{\chi(\tau) \choose k-l}\\
&= -\sum_{l=1}^{k-1} {\chi \choose l}\cup {\chi \choose k-l}(\sigma,\tau), \;\forall \sigma,\tau\in G,
\end{aligned}
\]
as desired.
\end{proof}
\begin{cor}
\label{cor:aaaa}
 Let $\chi$ be an element in $H^1(G,\F_p)={\rm Hom}(G,\F_p)$. Let $k<p$ be a positive integer. Then the system 
\[
\sM =\left\{ -{\chi \choose j-i}\mid 1\leq i<j\leq k+1\right \}
\]
is a complete defining system for the $n$-tuple $(-\chi,-\chi,\ldots,-\chi)$ ($k$ copies of $-\chi$).
\end{cor}
\begin{proof} This follows immediately from Corollary~\ref{cor:d of binom}.
\end{proof}

\begin{prop}
\label{prop:baaa}
  Let $F$ be a field containing a primitive $p$-th root of unity. Let $a$ and $b$ be elements which are linearly independent modulo $(F^\times)^p$. Let $k<p$ be a positive integer. Assume that $\chi_a\cup\chi_b=0$. Then the $(k+1)$-fold Massey product $\langle -\chi_b, -\chi_a,\ldots,-\chi_a\rangle$ ($k$ copies of $-\chi_a$) is defined and has a complete defining system of the form $\sM=\{a_{ij}\in C^1(G,\F_p)\mid 1\leq i<j\leq k+2\}$, where
\[
a_{ij}=-{\chi_a \choose j-i}, \;\text{ for all } 2\leq i<j\leq k+2.
\]
\end{prop}
\begin{proof}
By Corollary~\ref{cor:aaaa}, the system $\{a_{ij}:=-{\chi_a \choose j-i}\mid 2\leq i<j\leq k+2\}$ is a complete defining system for the $k$-tuple $(-\chi_a,\ldots,-\chi_a)$. We set $a_{1,2}=-\chi_b$. We shall prove by induction on $j=3,4,\ldots,k+2$ that there exist $a_{13},a_{14},\ldots,a_{1j}\in C^1(G_F,\F_p)$ such that
\[
d a_{1 r}= \sum_{l=2}^{r-1} a_{2l}\cup a_{l r}, \forall r= 3,4,\ldots,j.
\]
Then this will imply immediately that the system $\{ a_{ij}\mid 1\leq i<j\leq k+2\}$ is a complete defining sytem for the $(k+1)$-tuple $(-\chi_b,-\chi_a,\ldots,-\chi_a)$.

If $j=3$, then since $\chi_a\cup\chi_b=0$, there exists $\varphi_{ba}\in C^1(G_F,\F_p)$ such that $d \varphi_{ba}=\chi_b\cup \chi_a$. We set $a_{1,3}:=\varphi_{ba}$. Then 
\[
d a_{1,3}= \chi_b\cup\chi_a= a_{1,2}\cup a_{2,3}.
\]

Now assume that $j>3$. By induction hypothesis there exist $b_{13},b_{14},\ldots,b_{1j}\in C^1(G_F,\F_p)$ such that
\[
d b_{1 r}= \sum_{l=2}^{r-1} a_{1l}\cup a_{l r}, \forall r= 3,4,\ldots,j.
\]

Then the system
\[
\sN:=\{a_{12}, b_{13},\ldots, b_{1j}, \text{ and }a_{\mu\nu} \text{ with } 2\leq \mu<\nu \leq j+1\}
\]
is a defining system  for the $j$-fold Massey product $\langle -\chi_b,-\chi_a,\ldots,-\chi_a\rangle$. Then the value of this Massey product with respect to $\sN$ is
\[
\begin{aligned}
\langle \chi_b,\chi_a,\ldots,\chi_a\rangle_{\sN}&= a_{12}\cup a_{2,j+1}+b_{13}\cup a_{3,j+1}+\cdots+b_{1j}\cup a_{j,j+1}\\
&=-\chi_b\cup -{\chi_a\choose j-1} + b_{13}\cup -{\chi_a\choose j-2}+\cdots+b_{1j}\cup -{\chi_a\choose 1}.
\end{aligned}
\]

For the ease of notation, we denote $\res_{\ker\chi_a}$ just by $\res_a$. Then we have
\[
\begin{aligned}
&\res_a(\langle-\chi_b, -\chi_a,\ldots,-\chi_a\rangle_{\sN})\\
&= \res_a(-\chi_b)\cup \res_a(-{\chi_a\choose j-1}) + \res_a(b_{13})\cup \res_a(-{\chi_a\choose j-2}) + \cdots + \res_a(b_{1j}) \cup \res_a(-{\chi_a\choose 1}) \\
&= \res_a(-\chi_b)\cup 0 + \res_a(b_{13})\cup 0 + \cdots + \res_a(b_{1j}) \cup 0=0.
\end{aligned}
\]
Then by Proposition \ref{prop:killed by p-cyclic extension}, there exists $x\in F^\times$ such that
\[
-\chi_b\cup -{\chi_a\choose j-1} + b_{13}\cup -{\chi_a\choose j-2}+\cdots+b_{1j}\cup -{\chi_a\choose 1}= -\chi_x\cup -{\chi_a\choose 1} \in H^2(G,\F_p).
\]
Hence there exists $a_{1,j+1}$ in $C^1(G,\F_p)$ such that
\[
\chi_b\cup -{\chi_a\choose j-1} + b_{13}\cup -{\chi_a\choose j-2}+\cdots+b_{1,j-1}\cup -{\chi_a\choose 2}+ (b_{1j}+\chi_x)\cup-{\chi_a\choose 1}=d a_{1,j+1}.
\]
We set $a_{1r}:=b_{1r}$ for $r=3,4,\ldots,j-1$, and $a_{1j}:= b_{1j}+\chi_x$. Then we have 
\[
d a_{1 r}= d b_{1r}=\sum_{l=2}^{r-1} a_{1l}\cup a_{l r}, \forall r= 3,4,\ldots,j,
\]
and 
\[
d a_{1,j+1 }= \sum_{l=2}^{j} a_{1l}\cup a_{l ,j+1},
\]
as desired.
\end{proof}
\begin{prop}
\label{prop:abaa}
 Let $F$ be a field containing a primitive $p$-th root of unity. Let $a$ and $b$ be elements which are linearly independent modulo $(F^\times)^p$. Let $k<p$ be a positive integer. Assume that $\chi_a\cup\chi_b=0$. Then the $(k+2)$-fold Massey product $\langle -\chi_a,-\chi_b, -\chi_a,\ldots,-\chi_a\rangle$ ($k+1$ copies of $-\chi_a$) is defined and has a complete defining system of the form $\sM=\{a_{ij}\in C^1(G,\F_p)\mid 1\leq i<j\leq k+3\}$, where
\[
a_{ij}=-{\chi_a \choose j-i}, \;\text{ for all } 3\leq i<j\leq k+3.
\]
\end{prop}

\begin{proof}
By Proposition~\ref{prop:baaa}, there exists a system $\{a_{ij}\in C^1(G,\F_p)\mid 2\leq i<j\leq k+3\}$ such that it is a complete defining system for the $k+1$-tuple $(-\chi_b,-\chi_a,\ldots,-\chi_a)$ and that
\[
a_{ij}=-{\chi_a \choose j-i}, \;\text{ for all } 3\leq i<j\leq k+3.
\]
 We set $a_{1,2}=-\chi_a$. We shall prove by induction on $j=3,4,\ldots,k+3$ that there exist $a_{13},a_{14},\ldots,a_{1j}\in C^1(G_F,\F_p)$ such that
\[
d a_{1 r}= \sum_{l=2}^{r-1} a_{2l}\cup a_{l r}, \forall r= 3,4,\ldots,j.
\]
Then this will imply immediately that the system $\{ a_{ij}\mid 1\leq i<j\leq k+3\}$ is a complete defining system for the $(k+2)$-tuple $(-\chi_a,-\chi_b,-\chi_a,\ldots,-\chi_a)$.

If $j=3$, then since $\chi_a\cup\chi_b=0$, there exists $\varphi_{ab}\in C^1(G_F,\F_p)$ such that $d \varphi_{ab}=\chi_a\cup \chi_b$. We set $a_{1,3}:=\varphi_{ab}$. Then 
\[
d a_{1,3}= \chi_a\cup\chi_b= a_{1,2}\cup a_{2,3}.
\]

Now assume that $j>3$. By induction hypothesis there exist $b_{13},b_{14},\ldots,b_{1j}\in C^1(G_F,\F_p)$ such that
\[
d b_{1 r}= \sum_{l=2}^{r-1} a_{1l}\cup a_{l r}, \forall r= 3,4,\ldots,j.
\]

Then the system
\[
\sN:=\{a_{12}, b_{13},\ldots, b_{1j}, \text{ and }a_{\mu\nu} \text{ with } 2\leq \mu<\nu \leq j+1\}
\]
is a defining system  for the $j$-fold Massey product $\langle -\chi_b,-\chi_a,\ldots,-\chi_a\rangle$. Then the value of this Massey product with respect to $\sN$ is
\[
\begin{aligned}
&\langle -\chi_a,-\chi_b,-\chi_a,\ldots,-\chi_a\rangle_{\sN}\\
&= a_{12}\cup a_{2,j+1}+b_{13}\cup a_{3,j+1}+\cdots+b_{1j}\cup a_{j,j+1}\\
&=-\chi_a\cup a_{2,j+1}+b_{13}\cup -{\chi_a\choose j-1} + b_{13}\cup -{\chi_a\choose j-2}+\cdots+b_{1j}\cup -{\chi_a\choose 1}.
\end{aligned}
\]

For the ease of notation, we denote $\res_{\ker\chi_a}$ just by $\res_a$. Then we have
\[
\begin{aligned}
&\res_a(\langle-\chi_a, -\chi_b, -\chi_a,\ldots,-\chi_a\rangle_{\sN})\\
&= \res_a(-\chi_a)\cup \res_a(a_{2,j+1})+ \res_a(b_{13})\cup \res_a(-{\chi_a\choose j-2}) + \cdots + \res_a(b_{1j}) \cup \res_a(-{\chi_a\choose 1}) \\
&= 0\cup \res_a(a_{2,j+1}) + \res_a(b_{13})\cup 0 + \cdots + \res_a(b_{1j}) \cup 0=0.
\end{aligned}
\]
By Proposition \ref{prop:killed by p-cyclic extension}, there exists $x\in F^\times$ such that
\[
-\chi_a\cup a_{2,j+1}+ b_{13}\cup -{\chi_a\choose j-2}+\cdots+b_{1j}\cup -{\chi_a\choose 1}= -\chi_x\cup -{\chi_a\choose 1} \in H^2(G,\F_p).
\]
Hence there exists $a_{1,j+1}$ in $C^1(G,\F_p)$ such that
\[
-\chi_a\cup a_{2,j+1} + b_{13}\cup -{\chi_a\choose j-2}+\cdots+b_{1,j-1}\cup -{\chi_a\choose 2}+ (b_{1j}+\chi_x)\cup-{\chi_a\choose 1}=d a_{1,j+1}.
\]
We set $a_{1r}:=b_{1r}$ for $r=3,4,\ldots,j-1$, and $a_{1j}:= b_{1j}+\chi_x$. Then we have 
\[
d a_{1 r}= d b_{1r}=\sum_{l=2}^{r-1} a_{1l}\cup a_{l r}, \forall r= 3,4,\ldots,j,
\]
and 
\[
d a_{1,j+1 }= \sum_{l=2}^{j} a_{1l}\cup a_{l ,j+1},
\]
as desired.
\end{proof}

\begin{thm}
\label{thm:vanishing higher}
 Let $F$ be a field containing a primitive $p$-th root of unity. Let $a$ and $b$ be elements which are linearly independent modulo $(F^\times)^p$. Let $k<p$ be a positive integer. Assume that $\chi_a\cup\chi_b=0$. Then we have
\begin{enumerate}
\item The $(k+1)$-fold Massey products $\langle \chi_b, \chi_a,\ldots,\chi_a\rangle$ and $\langle \chi_a,\ldots,\chi_a,\chi_b\rangle$ ($k$ copies of $\chi_a$) are defined and contain 0.
\item The $(k+2)$-fold Massey products $\langle \chi_a,\chi_b, \chi_a,\ldots,\chi_a\rangle$ and $\langle \chi_a,\ldots,\chi_a,\chi_b,\chi_a\rangle$ ($k+1$ copies of $\chi_a$) are defined and contain 0.
\end{enumerate}
\end{thm}

\begin{proof}
We recall the following formal property of Massey products. If $\langle a_1,a_2,\ldots,a_n\rangle$ is defined, then $\langle a_n,a_{n-1},\ldots,a_1\rangle$ is defined and 
\[
\langle a_1,a_2,\ldots,a_n\rangle =\pm \langle a_n,a_{n-1},\ldots,a_1\rangle.
\]
(See \cite[Theorem 8]{Kra}.) Observe also that $-\chi_a=\chi_{a^{-1}}$ for every $a\in F^\times$.
The statement then follows from the two previous  propositions.
\end{proof}

\end{document}